\documentclass[11pt,A4]{article}
\usepackage[latin5]{inputenc}
\usepackage{amsmath,amssymb}
\usepackage{amsthm}
\usepackage{indentfirst}

\newtheorem{theorem}{Theorem}

\newtheorem{corollary}{Corollary}

\newtheorem{proposition}{Proposition}

\numberwithin{definition}{section} \numberwithin{theorem}{section}
\numberwithin{lemma}{section}\numberwithin{corollary}{section}
\numberwithin{equation}{section} \numberwithin{example}{section}
\numberwithin{proposition}{section} \numberwithin{remark}{section}
\begin{document}

\begin{center}

{\Large On the Robustness of the Integrable Trajectories of the Control Systems with Limited Control Resources}

\vspace{5mm}

\vspace{3mm}

Nesir Huseyin$^1$, Anar Huseyin$^2$,  Khalik G. Guseinov$^3$

\vspace{3mm}
{\small $^1$Cumhuriyet University, Faculty of Education, Department of Mathematics and \\ Science Education, 58140 Sivas, TURKEY

e-mail: nhuseyin@cumhuriyet.edu.tr

\vspace{2mm}

$^2$Cumhuriyet University, Faculty of Science, Department of Statistics and Computer Sciences \\ 58140 Sivas, TURKEY

e-mail: ahuseyin@cumhuriyet.edu.tr

\vspace{2mm}

$^3$Eskisehir Technical  University, Faculty of Science, Department of Mathematics \\ 26470 Eskisehir, TURKEY

e-mail: kguseynov@eskisehir.edu.tr}

\end{center}

\vspace{5mm}

\textbf{Abstract.}
The control system described by Urysohn type integral equation is considered where the system is nonlinear with respect to the phase vector and is affine with respect to the control vector. The control functions are chosen from the closed ball of the space $L_q\left(\Omega;\mathbb{R}^m\right),$ $q>1,$ with radius $r$ and centered at the origin. The trajectory of the system is defined as $p$-integrable multivariable function from the space $L_p\left(\Omega;\mathbb{R}^n\right),$ $\frac{1}{q}+\frac{1}{p}=1,$   satisfying the system's equation almost everywhere. It is shown that the system's trajectories are robust with respect to the remaining control resource. Applying this result it is proved that every trajectory can be approximated by the trajectory obtained by full consumption of the total control resource.

\vspace{3mm}

\vspace{3mm}

\section{Introduction}

The control systems described by integral equations is one of the important chapters of the control systems theory. The integral models undoubtedly have some advantages over differential ones, since the integral models allows to use continuous, and even integrable functions for the systems trajectory. It should be also underlined that the solutions concept for initial and boundary value problems for differential equations can be reduced to solution notion for appropriate integral equation. Note that the theory of the linear integral equations is considered one of the origins of the contemporary functional analysis (see, e.g. \cite{cor} - \cite{sha}).

The integral constraint on the control functions is inevitable, if the control resource is exhausted by consumption such as energy, fuel, finance and etc. The integral constraint on the control functions differs from geometric constraint, since the integrally constraint does not guarantee the geometric boundedness. Therefore, the control systems with integral constraints on the control functions have special behaviour and investigation of the systems requires different approaches (see, e.g. \cite{con} - \cite{sub}  and references therein).

The paper is organised as follows. In section 2 the basic conditions and preliminary propositions, which are used in following arguments, are given. In Section 3 it is proved that every trajectory is robust with respect to the remaining control resource (Theorem \ref{teo4.1}). It is also proved the set of trajectories coincides with the closure of the set of trajectories obtained by full consumption of the total control resource (Theorem \ref{teo4.2}).

\section{Preliminaries}

Consider the control system described by the  Urysohn type integral equation
\begin{eqnarray} \label{ue1}
\displaystyle x(\xi)=f\left(\xi,x\left(\xi\right)\right)+\lambda
\int_{\Omega}
\left[K_1\left(\xi,s,x\left(s\right)\right)+K_2\left(\xi,s,x\left(s\right)
\right) u\left(s\right)\right] ds
\end{eqnarray}
where  $x\in \mathbb{R}^n$ is the state vector, $u\in
\mathbb{R}^m$ is the control vector,  $\lambda \in \mathbb{R}^1,$ $\xi \in \Omega,$ $\Omega \subset \mathbb{R}^k$ is a compact set. Without loss of generality it will be assumed that $\lambda \geq 0.$

For given $q>1$ and $r > 0$ we denote
\begin{eqnarray*}
U_{q,r}=\left\{u(\cdot) \in L_q\big(\Omega;\mathbb{R}^m\big):
\left\| u(\cdot) \right\|_q \leq r \right\} ,
\end{eqnarray*}
where $L_q\big(\Omega;\mathbb{R}^m\big)$ is the space of Lebesgue measurable functions $u(\cdot):\Omega\rightarrow \mathbb{R}^m$ such that $\left\|u(\cdot)\right\|_q <+\infty,$ $\displaystyle \left\|u(\cdot)\right\|_q =\left(\int_{\Omega} \left\| u(s)\right\|^q ds\right)^{\frac{1}{q}},$
$\left\| \cdot \right\|$ denotes the Euclidean norm.

$U_{q,r}$ is called the set of admissible control functions and every  $u(\cdot) \in U_{q,r}$ is said to be an admissible control function.

For given $q\in (1,+\infty)$ let  $p\in (1,+\infty)$ be such that $\displaystyle \frac{1}{p}+\frac{1}{q}=1.$
It is assumed that the following conditions are satisfied.

\vspace{2mm}

\textbf{2.A.} The function  $f(\cdot,x ):\Omega\rightarrow
\mathbb{R}^{n}$ is Lebesgue measurable for every fixed $x\in \mathbb{R}^n$, $f(\cdot,0) \in L_{p}\left( \Omega ;\mathbb{R}^n\right)$ and there exists $\gamma_0(\cdot) \in L_{\infty}\left( \Omega ;\left[0,+\infty \right) \right)$ such that for almost all (a.a.) $\xi \in \Omega$ the inequality
\begin{eqnarray*}
\left\Vert f(\xi,x_{1})-f(\xi,x_{2})\right\Vert \leq \gamma_{0}(\xi) \left\Vert
x_{1}-x_{2}\right\Vert
\end{eqnarray*} is satisfied for every $x_1\in \mathbb{R}^n$ and $x_2\in \mathbb{R}^n$, where $L_{\infty}\big(\Omega;\mathbb{R}^{n_*}\big)$ is the space of Lebesgue measurable functions $g(\cdot):\Omega\rightarrow \mathbb{R}^{n_*}$ such that
$\left\|g(\cdot)\right\|_{\infty} <+\infty,$ $\displaystyle \left\|g(\cdot)\right\|_{\infty} =\inf \{ c>0:  \left\| g(s)\right\| \leq c \ \mbox{for a.a.} \ s \in \Omega \};$

\vspace{2mm}

\textbf{2.B.} The function  $K_1(\cdot,\cdot,x ): \Omega  \times \Omega \rightarrow
\mathbb{R}^{n}$ is Lebesgue measurable for every fixed $x\in \mathbb{R}^n$, $K_1(\cdot,\cdot,0) \in L_{p} \left(\Omega \times \Omega ;\mathbb{R}^n\right)$ and there exists $\gamma_1(\cdot,\cdot): \Omega \times \Omega \rightarrow \left[0,+\infty \right)$ such that
\begin{eqnarray*}
\int_{\Omega}\left( \int_{\Omega} \gamma_1(\xi,s)^q \, ds \right)^{\frac{p}{q}} d\xi < +\infty
\end{eqnarray*}
and
for a.a. $(\xi,s) \in  \Omega \times \Omega$ the inequality
\begin{eqnarray*}
\left\Vert K_1(\xi,s, x_{1})-K_1(\xi,s,x_{2})\right\Vert \leq \gamma_{1}(\xi,s) \left\Vert
x_{1}-x_{2}\right\Vert
\end{eqnarray*} is satisfied for every $x_1\in \mathbb{R}^n$ and $x_2\in \mathbb{R}^n$;

\vspace{2mm}

\textbf{2.C.} The function  $K_2(\cdot,\cdot,x ): \Omega \times \Omega \rightarrow
\mathbb{R}^{n\times m}$ is Lebesgue measurable for every fixed $x\in \mathbb{R}^n$, $K_2(\cdot,\cdot,0) \in L_{p}\left( \Omega \times \Omega ;\mathbb{R}^{n\times m}\right)$ and there exists $\gamma_2(\cdot,\cdot) \in L_{\infty}\left( \Omega \times \Omega;\left[0,+\infty \right) \right)$ such that for a.a. $(\xi,s) \in \Omega \times \Omega$ the inequality
\begin{eqnarray*}
\left\Vert K_2(\xi,s, x_{1})-K_2(\xi,s,x_{2})\right\Vert \leq  \gamma_{2}(\xi,s) \left\Vert x_{1}-x_{2}\right\Vert
\end{eqnarray*} is satisfied for every $x_1\in \mathbb{R}^n$ and $x_2\in \mathbb{R}^n$;

\vspace{2mm}

\textbf{2.D.} The inequality
\begin{eqnarray*}
\displaystyle 6^{p-1}\left[ \kappa_{0}^{p} +\lambda^{p} \kappa_{1}^{p} +\lambda^{p} r^{p} \kappa_{2}^{p} \mu(\Omega)\right]  < 1
\end{eqnarray*}
is satisfied, where $\mu(\Omega)$ denotes the Lebesgue measure of the set $\Omega$,
\begin{eqnarray}\label{kap0}
\kappa_0= \left\| \gamma_0(\cdot) \right\|_{\infty}, \ \ \kappa_2 =\left\|\gamma_2(\cdot,\cdot)\right\|_{\infty},
\end{eqnarray}
\begin{eqnarray}\label{kap1}
\displaystyle \kappa_1=\left( \int_{\Omega} \left(\int_{\Omega}\gamma_1(\xi,s)^{q} ds\right)^{\frac{p}{q}} \, d\xi \right)^{\frac{1}{p}}.
\end{eqnarray}

Denote
\begin{eqnarray}\label{alfa0}
\alpha_0= \left(\int_{\Omega}\left\|f(\xi, 0)\right\|^{p}  d\xi \right)^{\frac{1}{p}}, \ \ \alpha_i = \left(\int_{\Omega} \int_{\Omega}\left\|K_i(\xi, s,0)\right\|^{p} ds \, d\xi \right)^{\frac{1}{p}}, \ i=1,2
\end{eqnarray}
\begin{eqnarray}\label{ellam}
L_*(\lambda) =\displaystyle 6^{p-1}\left[ \kappa_{0}^{p} +\lambda^{p} \kappa_{1}^{p} +\lambda^{p} r^{p} \kappa_{2}^{p}\mu(\Omega)\right].
\end{eqnarray}

Condition 2.D implies that $L_*(\lambda)<1.$  Let us set
\begin{eqnarray*}\label{emlam*}
T_*(\lambda) = \displaystyle 6^{p-1}\left[ \alpha_{0}^{p} +\lambda^{p} \alpha_{1}^{p} \mu(\Omega)^{\frac{p}{q}} +\lambda^{p} \alpha_2^{p}r^{p} \right],
\end{eqnarray*}
\begin{eqnarray}\label{beta*}
\beta_*=\left[\frac{T_*(\lambda)}{1-L_*(\lambda)}\right]^{\frac{1}{p}}
\end{eqnarray} where $\alpha_0,$ $\alpha_1$ and $\alpha_2$ are defined by (\ref{alfa0}).

Let $u(\cdot)\in U_{q,r}$ be a given admissible control function. A function $x(\cdot) \in L_{p}\left( \Omega ; \mathbb{R}^n\right)$ satisfying the integral equation (\ref{ue1}) for a.a. $\xi \in \Omega$ is said to be a trajectory of the
system (\ref{ue1}) generated by the admissible control function
$u(\cdot)\in U_{q,r} \ .$ The set of trajectories of the system (\ref{ue1})
generated by all admissible control functions $u(\cdot)\in U_{q,r}$ is denoted by
$\mathbf{X}_{p,r}$ and is called the set of trajectories of the system (\ref{ue1}).

Now we will formulate some propositions the proofs of which are given in Huseyin (2020), and will be used in following arguments.
\begin{proposition} \label{prop2.1} \cite{hus} Every admissible control function $u(\cdot)\in U_{q,r}$ generates unique trajectory of the system (\ref{ue1}).
\end{proposition}

\begin{proposition}\label{prop2.2} \cite{hus} For each  $x(\cdot)\in \mathbf{X}_{p,r}$ the inequality
\begin{eqnarray*}\left\|x(\cdot)\right\|_{p} \leq \beta_*
\end{eqnarray*}
is satisfied where $\beta_*$ is defined by (\ref{beta*}).
\end{proposition}

\begin{proposition}\label{prop2.3} \cite{hus} The set of trajectories $\mathbf{X}_{p,r}$ is a compact and path-connected subset of the space $L_{p}\left( \Omega; \mathbb{R}^n\right).$
\end{proposition}

\section{Robustness of the Trajectories}

The robustness of the system with respect to some parameter which takes values on a given set, usually means that a variation of the parameter within a given set generates an insignificant deviation of the phase state.
The robustness of the trajectory of the control system with respect to the remaining control resource means that no matter how much the remaining management resource is, applying special method for complete consumption of the remaining control resource, it is possible to obtain a small variation of the  original trajectory.

Let
\begin{eqnarray}\label{c*}
c_*=2\lambda r\left[\frac{6^{p-1}}{1-L_*(\lambda)}\right]^{\frac{1}{p}}
\end{eqnarray} where $L_*(\lambda)$ is defined by (\ref{ellam}).

Denote $g_*(\xi,s)=K_2(\xi,s,0),$ $(\xi,s)\in \Omega \times \Omega.$ According to the condition 2.C we have
$g_*(\cdot,\cdot) \in L_{p}\left( \Omega \times \Omega ;\mathbb{R}^{n\times m}\right).$ It is known (see, \cite{kan}, p.318) that for given $\varepsilon >0$ there exists a continuous function $g_{\varepsilon}(\cdot,\cdot):\Omega \times \Omega \rightarrow \mathbb{R}^{n\times m}$ such that
\begin{eqnarray}\label{eq79}
\int_{\Omega}\int_{\Omega} \left\|g_*(\xi,s)-g_{\varepsilon}(\xi,s) \right\|^{p}  ds \, d\xi  \leq \frac{\varepsilon^{p}}{3c_{*}^{p}}.
\end{eqnarray} where $c_*$ is defined by (\ref{c*}).

Now, let us set
\begin{eqnarray}\label{eq80}
M(\varepsilon)=\max \left\{ \left\|g_{\varepsilon}(\xi,s) \right\|: (\xi,s)  \in  \Omega \times \Omega \right\}.
\end{eqnarray}

From Proposition \ref{prop2.3}, i.e. from the compactness of the set of trajectories $\mathbf{X}_{p,r}\subset L_{p}(\Omega; \mathbb{R}^n)$  it follows the validity of the following proposition.
\begin{proposition}\label{prop4.1} For every $\varepsilon >0$ there exists $\displaystyle \delta_*(\varepsilon) \in \left(0, \frac{\varepsilon^{p}}{3M(\varepsilon)^{p}c_*^{p}\mu(\Omega)}\right)$ such that for each Lebesgue measurable set $\Omega_* \subset \Omega$ such that
\begin{eqnarray*}
\displaystyle  \mu(\Omega_*) \leq \delta_*(\varepsilon),
\end{eqnarray*}
 the inequality
\begin{eqnarray*}
\displaystyle  \int_{\Omega_*} \left\|x(s)\right\|^{p}ds \leq \frac{\varepsilon^{p}}{3\kappa_{2}^{p}c_{*}^{p}\mu(\Omega)}
\end{eqnarray*}
is satisfied for every $x(\cdot) \in \mathbf{X}_{p,r}$ where $c_*$ is defined by (\ref{c*}).
\end{proposition}

\begin{theorem}\label{teo4.1}
Let $\varepsilon >0$ be a given number,  $x(\cdot)\in \mathbf{X}_{p,r}$ be a trajectory of the system (\ref{ue1}) generated by the admissible control function $u(\cdot)\in U_{q,r},$ $\left\|u(\cdot)\right\|_p =r_0 <r,$ $\Omega_*\subset \Omega$ be Lebesgue measurable set,
\begin{eqnarray}\label{eq71}
v(\xi)=\left\{
\begin{array}{llll}
u(\xi) & \mbox{if} & \xi \in \Omega \setminus \Omega_* \\
u_*(\xi) & \mbox{if} & \xi \in \Omega_*
\end{array}
\right.
\end{eqnarray}
be such that $\left\|v(\cdot)\right\|_{q}=r,$ $z(\cdot)\in \mathbf{X}_{p,r}$ be the trajectory of the system (\ref{ue1})
generated by the admissible control function $v(\cdot)\in U_{q,r}$.  If
\begin{eqnarray}\label{eq72}
\mu(\Omega_*)\leq  \delta_*(\varepsilon)
\end{eqnarray}
then
\begin{eqnarray*}
\left\|x(\cdot)-z(\cdot)\right\|_{p} \leq \varepsilon
\end{eqnarray*} where $\delta_*(\varepsilon)$ is defined in Proposition \ref{prop4.1}.
\end{theorem}

\begin{proof}  From (\ref{kap0}), (\ref{eq71}), Conditions 2.A, 2.B, 2.C, inclusion $u(\cdot)\in U_{q,r}$ and H\"{o}lder's inequality it follows that
\begin{eqnarray}\label{eq76}
\displaystyle && \left\|x(\xi)-z(\xi)\right\| \leq  \kappa_0 \left\|x(\xi)-z(\xi)\right\| + \lambda  \left(\int_{\Omega} \gamma_1 (\xi,s)^{q} ds \right)^{\frac{1}{q}}\left\|x(\cdot)-z(\cdot)\right\|_{p} \nonumber \\ && +\lambda \kappa_2 r  \left\|x(\cdot)-z(\cdot)\right\|_{p} + \lambda  \int_{\Omega_*} \left\|K_2(\xi,s,z(s))\right\| \left\| u(s)-v(s) \right\| ds
\end{eqnarray} for a.a. $\xi \in \Omega.$ (\ref{kap0}) and Condition 2.C imply
\begin{eqnarray}\label{eq77}
\int_{\Omega_*} \left\|K_2(\xi,s,z(s))\right\| \left\|u(s)-v(s)\right\| ds &\leq &  \kappa_2 \int_{\Omega_*} \left\|z(s) \right\| \left\|u(s)-v(s)\right\| ds \nonumber \\ &+& \int_{\Omega_*} \left\|K_2(\xi,s,0)\right\|\left\|u(s)-v(s)\right\| ds
\end{eqnarray}
for a.a. $\xi \in \Omega.$

Since $z(\cdot)\in \mathbf{X}_{p,r}$, then inclusions  $u(\cdot)\in U_{q,r}$, $v(\cdot)\in U_{q,r}$, (\ref{eq72}), Proposition \ref{prop4.1} and H\"{o}lder's inequality yield
\begin{eqnarray}\label{eq78}
\int_{\Omega_*} \left\|z(s) \right\| \left\|u(s)-v(s)\right\| ds &\leq & \left(\int_{\Omega_*} \left\|z(s)\right\|^{p}\right)^{\frac{1}{p}}\cdot \left(\int_{\Omega_*} \left\|u(s)-v(s)\right\|^{q} ds\right)^{\frac{1}{q}}
 \nonumber \\ &\leq& 2r \cdot \frac{\varepsilon}{c_* \kappa_2 \left[3\mu(\Omega)\right]^{\frac{1}{p}}} .
\end{eqnarray}
From (\ref{eq80}), inclusions $u(\cdot)\in U_{q,r},$ $v(\cdot)\in U_{q,r}$ and H\"{o}lder's inequality it follows that
\begin{eqnarray}\label{eq81}
&&\displaystyle \int_{\Omega_*} \left\|g_*(\xi,s) \right\| \cdot  \left\|u(s)-v(s)\right\| ds   \leq \int_{\Omega_*} \left\|g_*(\xi,s)-g_{\varepsilon} (\xi,s) \right\| \cdot  \left\|u(s)-v(s)\right\| ds \nonumber \\
\displaystyle  && \quad + \int_{\Omega_*} \left\|g_{\varepsilon}(\xi,s) \right\| \cdot  \left\|u(s)-v(s)\right\| ds \nonumber \\
\displaystyle && \leq  2r \left(\int_{\Omega} \left\|g_*(\xi,s)-g_{\varepsilon} (\xi,s) \right\|^{p}ds \right)^{\frac{1}{p}}+2r M(\varepsilon)\left[\mu(\Omega_*)\right]^{\frac{1}{p}}
\end{eqnarray}
for a.a. $\xi \in \Omega$ where $g_*(\cdot,\cdot)=K_2(\cdot,\cdot,0),$ $g_{\varepsilon}(\cdot,\cdot)$ is defined in (\ref{eq79}). Thus (\ref{eq77}), (\ref{eq78}) and (\ref{eq81}) yield
\begin{eqnarray*}
&& \int_{\Omega_*} \left\|K_2(\xi,s,z(s))\right\| \left\|u(s)-v(s)\right\| ds \leq 2r \cdot \frac{\varepsilon}{c_*\left[3\mu(\Omega)\right]^{\frac{1}{p}}}  \nonumber \\
&& \quad + 2r \left(\int_{\Omega} \left\|g_*(\xi,s)-g_{\varepsilon} (\xi,s) \right\|^{p}ds \right)^{\frac{1}{p}}+2r M(\varepsilon)\left[\mu(\Omega_*)\right]^{\frac{1}{p}} \end{eqnarray*} for a.a. $\xi \in \Omega.$ Finally, from the last inequality, (\ref{eq72}) and (\ref{eq76}) we obtain
\begin{eqnarray*}
\displaystyle && \left\|x(\xi)-z(\xi)\right\| \leq  \kappa_0 \left\|x(\xi)-z(\xi)\right\| + \lambda  \left(\int_{\Omega} \gamma_1 (\xi,s)^{q} ds \right)^{\frac{1}{q}}\left\|x(\cdot)-z(\cdot)\right\|_{p} \nonumber \\ && \quad +\lambda \kappa_2 r  \left\|x(\cdot)-z(\cdot)\right\|_{p} + 2\lambda r \cdot \frac{\varepsilon}{c_* \left[3\mu(\Omega)\right]^{\frac{1}{p}}}
\nonumber \\ && \quad + 2\lambda r \left(\int_{\Omega} \left\|g_*(\xi,s)-g_{\varepsilon} (\xi,s) \right\|^{p}ds \right)^{\frac{1}{p}}+2\lambda r M(\varepsilon)\left[\delta_*(\varepsilon)\right]^{\frac{1}{p}}
\end{eqnarray*} for a.a. $\xi \in \Omega$ and consequently
\begin{eqnarray*}
\displaystyle && \left\|x(\xi)-z(\xi)\right\|^{p} \leq 6^{p-1} \Big[ \kappa_0^{p} \left\|x(\xi)-z(\xi)\right\|^{p} + \lambda^{p}  \left(\int_{\Omega} \gamma_1 (\xi,s)^{q} ds \right)^{\frac{p}{q}}\left\|x(\cdot)-z(\cdot)\right\|_{p}^{p}
\nonumber \\ && \quad +\lambda^{p} \kappa_2^{p} r^{p}  \left\|x(\cdot)-z(\cdot)\right\|_{p}^{p} + 2^{p}\lambda^{p} r^{p} \cdot \frac{\varepsilon^{p}}{3\mu(\Omega)c_{*}^{p}}
\nonumber \\ && \quad + 2^{p}\lambda^{p} r^{p} \int_{\Omega} \left\|g_*(\xi,s)-g_{\varepsilon} (\xi,s) \right\|^{p}ds +2^{p}\lambda^{p} r^{p} M(\varepsilon)^{p}\delta_*(\varepsilon)\Big]
\end{eqnarray*} for a.a. $\xi \in \Omega.$ Integrating the last inequality on the set $\Omega$ and taking into consideration (\ref{kap1}), (\ref{ellam}), (\ref{eq79}) and the inequality
\begin{eqnarray*}
\displaystyle \delta_*(\varepsilon)<\frac{\varepsilon^{p}}{3M(\varepsilon)^{p}c_*^{p}\mu(\Omega)}
\end{eqnarray*}
we obtain
\begin{eqnarray*}
\displaystyle && \left\|x(\cdot)-z(\cdot)\right\|^{p}_{p} \leq 6^{p-1} \Big[ \kappa_0^{p} \left\|x(\cdot)-z(\cdot)\right\|^{p}_{p} + \lambda^{p} \int_{\Omega} \left(\int_{\Omega} \gamma_1 (\xi,s)^{q} ds \right)^{\frac{p}{q}} d\xi \cdot \left\|x(\cdot)-z(\cdot)\right\|_{p}^{p}
\nonumber \\ && \quad +\lambda^{p} \kappa_2^{p} r^{p} \mu(\Omega) \left\|x(\cdot)-z(\cdot)\right\|_{p}^{p} + 2^{p}\lambda^{p} r^{p} \cdot \frac{\varepsilon^{p}}{3c_{*}^{p}}
\nonumber \\ && \quad + 2^{p}\lambda^{p} r^{p} \int_{\Omega} \int_{\Omega} \left\|g_*(\xi,s)-g_{\varepsilon} (\xi,s) \right\|^{p}ds \, d\xi +2^{p}\lambda^{p} r^{p} M(\varepsilon)^{p}\mu(\Omega)\delta_*(\varepsilon)\Big] \nonumber \\
&&\leq 6^{p-1} \Big[ \kappa_0^{p} + \lambda^{p} \kappa_1^{p} +\lambda^{p} \kappa_2^{p} r^{p} \mu(\Omega)\Big] \cdot \left\|x(\cdot)-z(\cdot)\right\|_{p}^{p} \nonumber \\
&& \quad + 6^{p-1} \left[ 2^{p}\lambda^{p} r^{p} \cdot \frac{\varepsilon^{p}}{3c_{*}^{p}}
+ 2^{p}\lambda^{p} r^{p} \cdot \frac{\varepsilon^{p}}{3c_{*}^{p}}+2^{p}\lambda^{p} r^{p} M(\varepsilon)^{p}\mu(\Omega)\cdot \frac{\varepsilon^{p}}{3M(\varepsilon)^{p}c_*^{p}\mu(\Omega)}\right] \nonumber\\
&& = L_*(\lambda)\cdot \left\|x(\cdot)-y(\cdot)\right\|_{p}^{p} + 6^{p-1}\cdot \left[2\lambda r\right]^{p}\cdot \Bigg[ \frac{\varepsilon^{p}}{3c_{*}^{p}} +\frac{\varepsilon^{p}}{3c_{*}^{p}} +\frac{\varepsilon^{p}}{3c_{*}^{p}}\Bigg] \nonumber \\
&&= L_*(\lambda)\cdot \left\|x(\cdot)-y(\cdot)\right\|_{p}^{p} + 6^{p-1}\cdot \left[2\lambda r\right]^{p}\cdot \frac{\varepsilon^{p}}{c_{*}^{p}}
\end{eqnarray*} and finally, by virtue of (\ref{c*})
\begin{eqnarray*}
\displaystyle \left\|x(\cdot)-z(\cdot)\right\|_{p} \leq 2\lambda r\left[\frac{6^{p-1}}{1-L_*(\lambda)}\right]^{\frac{1}{p}}
\cdot \frac{\varepsilon}{c_{*}}=\varepsilon.
\end{eqnarray*}
The proof is completed.
\end{proof}

Denote
\begin{eqnarray*}
U_{q,r}^{*}=\left\{u(\cdot) \in L_q\big(\Omega;\mathbb{R}^m\big):
\left\| u(\cdot) \right\|_{q} = r \right\} ,
\end{eqnarray*}
and let $\mathbf{X}_{p,r}^{*}$ be the set of trajectories of the system (\ref{ue1})
generated by all admissible control functions $u(\cdot)\in U_{q,r}^{*}.$

\begin{theorem}\label{teo4.2} The equality
\begin{eqnarray*}
\mathbf{X}_{p,r}= cl \, \left(\mathbf{X}_{p,r}^{*}\right)
\end{eqnarray*} is satisfied where $cl$ denotes the closure of a set.
\end{theorem}

\begin{proof}
Let us choose an arbitrary $x(\cdot)\in \mathbf{X}_{p,r}$ generated by the control function $u(\cdot)\in U_{p,r}$ and let
$\left\|u(\cdot)\right\|_p=r_0<r$. For given $\varepsilon >0$ choose Lebesgue measurable set $\Omega_* \subset \Omega$ where $\mu(\Omega_*)\leq \delta_*(\varepsilon)$ and $\delta_*(\varepsilon)>0$ is defined as in Theorem \ref{teo4.1}. Assume that
\begin{eqnarray*}
\int_{\Omega \setminus \Omega_*}\left\|u(s)\right\|^p ds =r_1^p \, .
\end{eqnarray*} It is obvious that $r_1\leq r_0.$ Define control function
\begin{eqnarray*}
u_*(s)=\left\{
\begin{array}{llll}
u(s) \ , & \mbox{if} & s\in \Omega \setminus \Omega_* \ , \\
\displaystyle \left[\frac{r^p -r_1^p}{\mu(\Omega_*)} \right]^{\frac{1}{q}} \cdot b_* \ , & \mbox{if} & s\in \Omega_*
\end{array}
\right.
\end{eqnarray*} where $b_*\in \mathbb{R}^m$ is an arbitrary vector such that $\left\| b_* \right\|=1.$ One can show that $u_*(\cdot)\in U_{p,r}^{*}.$ Let $x_*(\cdot):\Omega \rightarrow \mathbb{R}^m$ be the trajectory of the system (\ref{ue1}) generated by the control function $u_*(\cdot)\in U_{p,r}^{*}.$ Then $x_*(\cdot)\in \mathbf{X}_{p,r}^{*}$ and according to the theorem \ref{teo4.1} we have
\begin{eqnarray*}
\left\|x(\cdot)-x_*(\cdot)\right\|_{p} \leq \varepsilon
\end{eqnarray*} which implies that
\begin{eqnarray}\label{inc1}
x_*(\cdot) \in \mathbf{X}_{p,r}^{*} +\varepsilon B_{L_{p}}(1)
\end{eqnarray} where
\begin{eqnarray*}
B_{L_{p}}(1)=\left\{y(\cdot)\in L_{p}\left( \Omega; \mathbb{R}^m \right): \left\|y(\cdot)\right\|_{p} \leq 1\right\}.
\end{eqnarray*}

Since $\varepsilon >0$ is an arbitrarily chosen number, then (\ref{inc1}) yields that
\begin{eqnarray*}
x_*(\cdot) \in cl \, \left(\mathbf{X}_{p,r}^{*}\right)
\end{eqnarray*} and hence
\begin{eqnarray}\label{inc2}
 \mathbf{X}_{p,r} \subset cl \, \left(\mathbf{X}_{p,r}^{*}\right).
\end{eqnarray}

Since $cl \, \left(\mathbf{X}_{p,r}^{*}\right)  \subset \mathbf{X}_{p,r},$ then the inclusion (\ref{inc2}) completes the proof of the theorem.
\end{proof}

\begin{corollary} The equality
\begin{eqnarray*}
h_{L_{p}} \left(\mathbf{X}_{p,r}, \mathbf{X}_{p,r}^{*}\right)=0
\end{eqnarray*} is held where $h_{L_{p}}(\cdot,\cdot)$ stands for Hausdorff distance between the subsets of the space $L_{p}\left( \Omega;\mathbb{R}^n\right).$
\end{corollary}

\end{document}